\documentclass[12pt]{report}
\usepackage[utf8]{inputenc}

\title{EXPANSIONS OF REAL NUMBERS IN NON-INTEGER BASES AND CHARACTERISATION OF LAZY EXPANSION OF 1}
\author{Vorashil Farzaliyev \\ Supervisor: Dr Tom Kempton}
\date{Completed June 2018, Published in November 2024}

\usepackage{amsmath}
\usepackage{amsfonts}
\usepackage{amssymb}
\usepackage{stix}
\usepackage{amsthm}
\usepackage[makeroom]{cancel}
\usepackage{pgfplots}
\usepackage{fancybox}
\usepackage{float}
\usepackage{graphicx}
\usepackage{color}
\usepackage{multicol}
\usepackage{fancyhdr} 
\usepackage[numbers,square]{natbib}
\usepackage{mdframed}
\usepackage[nottoc]{tocbibind}
\usepackage[margin=1.5in]{geometry}
\usepackage{sectsty}
\usepackage{hyperref}
\hypersetup{
    colorlinks=true,
    linkcolor=blue,
    urlcolor=red,
    citecolor=blue,
    linktoc=all
}
\chapternumberfont{\huge} 
\chaptertitlefont{\large}
\newtheorem{lemma}{Lemma}[chapter]
\newtheorem{definition}{Definition}[chapter]
\newtheorem{theorem}{Theorem}[chapter]
\newtheorem{algorithm}{Algorithm}[chapter]

\begin{document}

\maketitle
\newpage

\chapter*{Acknowledgements}
\ \ I would like to sincerely thank my advisor Dr. Tom Kempton, who guided me throughout this year on this project and motivated me to work harder by posing interesting questions for me to think about during each of our meetings. This project would not have been possible without his valuable comments and explanations of the key concepts. I have always felt comfortable asking him my questions, both on and off the project, some of which related to an academic career in mathematics.
\newline
\ \ I would also like to express my gratitude to my family, whose unparalleled support played a pivotal role in helping me achieve a full scholarship to study mathematics abroad and made my education at the University of Manchester possible.
\newpage

\chapter*{Preface}
\ \ This publication represents my undergraduate dissertation, completed during my studies at the University of Manchester. I am publishing it now, years later, following the advice of a fellow student who suggested that presenting the solution for the characterization of 1 would be of interest to others. This work, which I should have published earlier, offers insights into a topic I greatly enjoyed exploring during my undergraduate years. It also provides the solution to the first open problem in \cite{erdos1990characterization}, known as the characterization of the expansion of 1 (i.e. \textit{Problem 1: Characterize the lazy expansions of 1.})
\newpage

\begin{center}
\hspace{0pt}
\vfill
\textit{I dedicate this paper to my grandmother Zarifa Rahimova, whom I have lost in the past years. She was the one who taught me counting and arithmetic, and her teachings are the main foundation of all my current mathematical knowledge.}
\vfill
\hspace{0pt}
\end{center}
\newpage
\chapter*{Abstract}
In this paper, our main focus is expressing real numbers on the non-integer bases. We denote those bases as $\beta$'s, which is also a real number and $\beta \in (1,2)$. This project has 3 main parts. The study of expansions of real numbers in such bases and algorithms for generating them will contribute to the first part of the paper. In this part, firstly, we will define those expansions as the sums of fractions with $1$'s or $0$'s in the nominator and powers of $\beta$ in the denominator. Then we will focus on the sequences of $1$'s and $0$'s generated by the nominators of in the sums we mentioned above. Such sequences will be called \textit{coefficient sequences} throughout the paper. In the second half, we will study the results in the first chapter of  \cite{erdos1990characterization}, namely the greedy and lazy $\beta $-expansions . The last part of the paper will be on the characterisation of lazy expansion of 1, which was the first open question at the end of \textit{Erdos and Komornik}. I still don't know if that problem has been solved already. However, the solution that was presented here is the original work of mine.
\newpage
\tableofcontents
\newpage
\chapter*{Introduction}

\ \ We can express real numbers in many ways. One of which is the binary expansion of real numbers. For example, for all $x \in (0,1)$, we can have $x = \sum_{i=1}^{\infty}\frac{d_i}{2^i}$, where $d_i \in \{0,1\}$. We know that any real number can have at most 2 different binary expansions. Moreover, one can consider expansions in different bases. e.g ternary, decimal, etc. 
\newline
\ \ As one can see, these all are expansions of real numbers on the \textit{integer} bases. Our main focus on this paper will be expansions in the \textit{non-integer} bases. We will call them $\beta$-expansions, where $\beta \in (0,1)$. What makes $\beta$-expansions particularly interesting is the fact that a real number on such bases can have uncountably many different expansions. We will show this result in the Chapter N.
\newline
\ \ Let $\beta \in (0,1)$ be given, and consider the expansion of $x$
\[
x = \sum_{i=1}^{\infty}a_i\beta^{-i} \quad where \ \ a_i \in \{0,1\}
\]
\ \ We know that the smallest $x$ for which we can have $\beta$-expansion is $0$. That is $0 = \frac{0}{\beta}+\frac{0}{\beta^2}+\frac{0}{\beta^3}+\frac{0}{\beta^4}+\dots$. Similarly the biggest $x$ is $\frac{1}{\beta-1} = \frac{1}{\beta}+\frac{1}{\beta^2}+\frac{1}{\beta^3}+\frac{1}{\beta^4}$, which comes from the sum of the geometric series.
\newline
\ \ In the next chapters, we will study algorithms for generating $\beta$-expansions, and we will show that for all $x \in [0,\frac{1}{\beta-1}]$, we can generate $\beta$-expansions.
\chapter{Algorithms for generating $\beta$-expansions}
In this chapter, we will study algorithms for generating $\beta$-expansions. We will also show that using each algorithm we can expand all real numbers $x \in [0,\frac{1}{\beta-1}]$
\section{Algorithm 1}
Let  $\beta \in (1,2)$ and $x \in [0,\frac{1}{\beta-1}]$ be given. To find an expansion of the form  
\[
x=\frac{1}{\beta^{m_1}}+\frac{1}{\beta^{m_2}}+\frac{1}{\beta^{m_3}}+\dots+\frac{1}{\beta^{m_n}}+\dots
\]
where ${m_i} \in \mathbb{N}$. We use Algorithm 1 to find these $m_i$'s. To save the space we will denote $\underbrace{T(T(T\dots}_\text{n times}(x)))$ as $T^{(n)}(x)$ where $n \in \mathbb{N}$.
\newline
\newline
\begin{algorithm}
First we define: $$T(x)=   \left\{
\begin{array}{ll}
      x-\frac{1}{\beta} &  if \quad x>\frac{1}{\beta} \\
      x-\frac{1}{\beta^2} & if \quad  \frac{1}{\beta}>x>\frac{1}{\beta^2} \\
      \dots \\
      x-\frac{1}{\beta^n} & if \quad \frac{1}{\beta^{n-1}}>x>\frac{1}{\beta^n} \\
      \dots \\
      x & if \quad x=0
\end{array} 
\right. $$
\newline
\textbf{STEP 1}: Find $x-T(x)$. 
\newline
In the trivial case, if $x-T(x)=0$, then return $x = 0$ by the definition of $T(x)$. Otherwise, if $x-T(x)>0$, 
\[
\frac{1}{\beta^{m_1}} = x - T(x)
\]
and proceed to the next step.
\newline
\textbf{STEP 2}: Find $T(x)-T^{(2)}(x)$.
\newline
If $T(x)-T^{(2)}(x) = 0$, then return $x = \frac{1}{\beta^{m_1}}$. Otherwise, if $T(x)-T^{(2)}(x) > 0$, 
\[
\frac{1}{\beta^{m_2}} = T(x)-T^{(2)}(x)
\]
and proceed to the next step.
\newline
\textbf{STEP k ($k>2$)}: Find $T^{(k-1)}(x)-T^{(k)}(x)$.
\newline
If $T^{(k-1)}(x)-T^{(k)}(x) = 0$, then return $x = \frac{1}{\beta^{m_1}}+\frac{1}{\beta^{m_2}}+\frac{1}{\beta^{m_3}}+\dots +\frac{1}{\beta^{m_{k-1}}}$. Otherwise, if $T^{(k-1)}(x)-T^{(k)}(x) > 0$, 
\[
\frac{1}{\beta^{m_k}} = T^{(k-1)}(x)-T^{(k)}(x)
\]
and proceed to the next step.
\end{algorithm}
As we see from above, our algorithm terminates and returns finite sum if for some $n \in \mathbb{N}$, $T^{(n)}(x) - T^{(n+1)}(x) = 0$. By the definition of $T(x)$, this is only true, i.e $T^{(n+1)}(x)=T(T^{(n)}(x))=T^{(n)}(x)$ if and only if $T^{(n)}(x)=0$. Hence, we conclude that the $\beta$-expansion of given $x$ is finite if and only if there exist $n \in \mathbb{N}$ such that $T^{(n)}(x)=0$. In that case, the expansion of given $x$ is as follows
\[
x = x = \frac{1}{\beta^{m_1}}+\frac{1}{\beta^{m_2}}+\frac{1}{\beta^{m_3}}+\dots +\frac{1}{\beta^{m_{n}}}
\]
Firstly, we will show that when the expansion of $x$ generated by our algorithm is finite, it's always equal to $x$.
\begin{lemma}
Let the expansion of $x$ given as $\frac{1}{\beta^{m_1}}+\frac{1}{\beta^{m_2}}+\frac{1}{\beta^{m_3}}+\dots +\frac{1}{\beta^{m_{n}}}$ and generated by \textit{Algorithm 1} be finite. Then 
\[
x = \frac{1}{\beta^{m_1}}+\frac{1}{\beta^{m_2}}+\frac{1}{\beta^{m_3}}+\dots +\frac{1}{\beta^{m_{n}}}
\]

\end{lemma}
\begin{proof}
We will use argument by contradiction to show the statement is true. Firstly assume that,
\[
x > \frac{1}{\beta^{m_1}}+\frac{1}{\beta^{m_2}}+\frac{1}{\beta^{m_3}}+\dots +\frac{1}{\beta^{m_{n}}}
\]
By Algorithm 1, $\frac{1}{\beta^{m_i}} =T^{(i-1)}(x)-T^{(i)}(x)$ for all $i \in \{1,2,3,\dots,n\}$ and here $T^{(0)}(x)$ denotes $x$. Then we can write the inequality above as follows

\begin{equation*} 
\begin{split}
x &> [x - T(x)] + [T(x)-T^{(2)}(x)] + [T^{(2)}(x)-T^{(3)}(x)]+ \dots + [T^{(n-1)}(x)-T^{(n)}(x)]\\
& = x - \cancel{T(x)} + \cancel{T(x)}-\cancel{T^{(2)}(x)} + \cancel{T^{(2)}(x)}-\cancel{T^{(3)}(x)}+ \dots + \cancel{T^{(n-1)}(x)}-T^{(n)}(x)\\
& = x - T^{(n)}(x)
\end{split}
\end{equation*}
Since, the smallest element in our sum is $\frac{1}{\beta^{m_n}}$, by the discussion above we know that $T^{(n)}(x)=0$. Hence, we conclude that $x>x$, which is a contradiction. Therefore our assumption that $x > \frac{1}{\beta^{m_1}}+\dots +\frac{1}{\beta^{m_{n}}}$ is false. 
\newline
Similarly we can show that $x < \frac{1}{\beta^{m_1}}+\dots +\frac{1}{\beta^{m_{n}}}$ is false. Following the method above we can show that, $x < \frac{1}{\beta^{m_1}}+\dots +\frac{1}{\beta^{m_{n}}}$ implies $x<x$, which is also a contradiction.
\newline
Hence, we conclude that 
\[
x = \frac{1}{\beta^{m_1}}+\frac{1}{\beta^{m_2}}+\frac{1}{\beta^{m_3}}+\dots +\frac{1}{\beta^{m_{n}}}
\]
as it's stated in the lemma.
\end{proof}
Now, we need to show that if the expansion generated by \textit{Algorithm 1} is infinite, it converges to $x$. Following lemma shows the result.
\begin{lemma}
Let the expansion of $x$ given as $\sum_{i=1}^{\infty}\frac{1}{\beta^{m_i}}$ and generated by Algorithm 1 be infinite. Then $x= \sum^{\infty}_{i=1}\frac{1}{\beta^{m_i}}$
\end{lemma}
\begin{proof}
Since the expansion is not finite, there does not exist $n\in \mathbb{N}$ such that $T^{(n)}(x)=0$. Also, by the definition of $T(x)$ it is never negative. Hence $T^{(n)}(x)>0$ for all $n \in \mathbb{N}$
\newline
To show the convergence, first we will use this fact. That is  $T^{(n)}(x)>0$ for all $n \in \mathbb{N}$. Equivalently,
\[
x > x - T^{(n)}(x) \quad \textit{for all} \ \ n \in \mathbb{N}
\]
This is also equivalent 
\begin{equation*} 
\begin{split}
x &> x - T^{(n)}(x)\\
&=[x - T(x)] + [T(x)-T^{(2)}(x)] + [T^{(2)}(x)-T^{(3)}(x)]+ \dots + [T^{(n-1)}(x)-T^{(n)}(x)]\\
& = \frac{1}{\beta^{m_1}}+\frac{1}{\beta^{m_2}}+\frac{1}{\beta^{m_3}}+\dots +\frac{1}{\beta^{m_{n}}}\\
\end{split}
\end{equation*}
for all $n \in \mathbb{N}$.
\newline
This show that, for all $n \in \mathbb{N}$, the sum of fractions $\frac{1}{\beta^{m_1}}+\frac{1}{\beta^{m_2}}+\frac{1}{\beta^{m_3}}+\dots +\frac{1}{\beta^{m_{n}}}$ does not exceed $x$.
\newline
Now, we need to show that as we add more elements to the sum of fractions, it gets closer to $x$. That is true if  $\lim_{n\to\infty} T^{(n)}(x)=0$, because $x-T^{(n)}(x)$ is getting closer to $x$, as $T^{(n)}(x) \to 0$. Therefore, we prove the lemma, if we show $\lim_{n\to\infty} T^{(n)}(x)=0$.
\newline
To do this we first will show that $T^{(n)}(x)$ is convergent as $n \to \infty$. Then we will show that, it is actually converging to $0$. If we can prove that $T^{(n)}(x)$ is decreasing and bounded as $n \to \infty$, we are done with the proof of convergence of $T^{(n)}(x)=0$
\newline 
Now, we will show $T^{(n)}(x)$ is decreasing as $n \to \infty$. That is 
\begin{align*}
 \forall n \in \mathbb{N} \quad  T^{(n-1)}(x)>T^{(n)}(x)\quad & {\Leftrightarrow} \quad \forall n \in \mathbb{N},\quad T^{(n-1)}(x)>T^{(n-1)}(x)-\frac{1}{\beta^{m_n}} \\
                             & {\Leftrightarrow}\quad \forall n \in \mathbb{N}, \quad  \frac{1}{\beta^{m_n}}>0\\
\end{align*}
Since $\frac{1}{\beta^{m_n}}$ is always positive for all $n \in \mathbb{N}$, and by the equivalence of the statements above, we conclude that $T^{(n)}(x)$ is decreasing as $n \to \infty$.
\newline
\ \ Using the fact that $T^{(n)}(x)$ is decreasing as $n \to \infty$, we can show $T^{(n)}(x)$ is bounded above by $x$ for all $n \in \mathbb{N}$, because $x>T(x)$ and $T(x)>T^{(2)}(x)>T^{(3)}(x)>\dots$. We also showed that $T^{(n)}(x)>0$ for all $n \in \mathbb{N}$, as the expansion is infinite. Hence, we can conclude that 
\[
0 < T^{(n)}(x) < x,\quad \forall n \in \mathbb{N}.
\]

As $T^{(n)}(x)$ is bounded and decreasing, we conclude that it's convergent. To show $\lim_{n\to\infty} T^{(n)}(x)=0$, let's assume by contradiction that it converges to a positive number. Since, $T^{(n)}(x)>0$ for all $n \in \mathbb{N}$ it cannot converge to a negative number. That is,
\[
\lim_{n\to\infty} T^{(n)}(x)=a>0
\]
for some $a \in (0,x)$.
\newline
But we can find $k \in \mathbb{N}$ such that $a>\frac{1}{\beta^{k}}>0$. That is $a>T^{(k-1)}(x)-T^{(k)}(x)>0$.
\newline
By Sandwich Theorem for limits, $$\lim_{k\to\infty}a>\lim_{k\to\infty} T^{(k-1)}(x)-\lim_{k\to\infty} T^{(k)}(x)>\lim_{k\to\infty}0.$$
\newline
That is, $$a>a-a>0.$$
\newline
So $a>0>0$ is a contradiction we get from our assumption. Hence $$\lim_{n\to\infty} T^{(n)}(x)=0.$$
\newline
The line above, finishes the proof of this lemma.So an infinite expansion of $x$ generated by \textit{Algorithm 1} always converge to $x$.
\end{proof}
\section{Algorithm 2}
Now we know that, we can always have $\beta$-expansion for given $x \in [0,\frac{1}{\beta-1}]$.
We might also be interested in the coefficient sequence $\{a_i\}_{i=1}^{\infty} $for given $x$ generated by $\beta$ expansion ($x=\sum_{i=1}^{\infty} a_i \beta^{-i}$, where $a_i \in \{0,1\}$). In our previous algorithm, we neglected all terms with zero coefficients, but using Algorithm 2, we can collect all coefficients in our expansion. This is particularly useful, if we want to express $x$ on the non-integer bases, like $\beta$. This allows us to study more interesting behaviours of real numbers on such bases.

\begin{algorithm} 
Let $x \in [0,\frac{1}{\beta-1}]$ and $\beta \in (1,2)$ be given such that $$x = \sum_{i=1}^{\infty}{a_i}{\beta^{-i}} \quad \text{where} \quad a_i \in \{0,1\}.$$
\newline
To find $a_i$'s, first we define a new function 
\[T^{(i)}(x)=   \left\{
\begin{array}{ll}
      \beta T^{(i-1)}(x)-1 &  T^{(i-1)}(x)\geq \frac{1}{\beta} \\
      \beta T^{(i-1)}(x) &  T^{(i-1)}(x)< \frac{1}{\beta} \\
\end{array} 
\right.
\]for $i\geq 1$ and where $T^{(i)}(x) = \underbrace{(T \circ T \circ \dots \circ T)}_{\text{i times}}(x)$ with $T^{(0)}(x)=x$.
\newline
Then, to find $a_i$, we simply calculate $$a_i=\beta T^{(i-1)}(x)-T^{(i)}(x)$$.
That is
\[a_i=   \left\{
\begin{array}{ll}
     0 & \textit{ if }T^{(i)}(x)=\beta T^{(i-1)}(x)\\
     1   &\textit{ if }T^{(i)}(x)=\beta T^{(i-1)}(x)-1 \\
\end{array} 
\right.
\]

\end{algorithm}
Here is the graph of $T(x)$.
\begin{figure}[H]
\centering
\includegraphics[width=10cm]{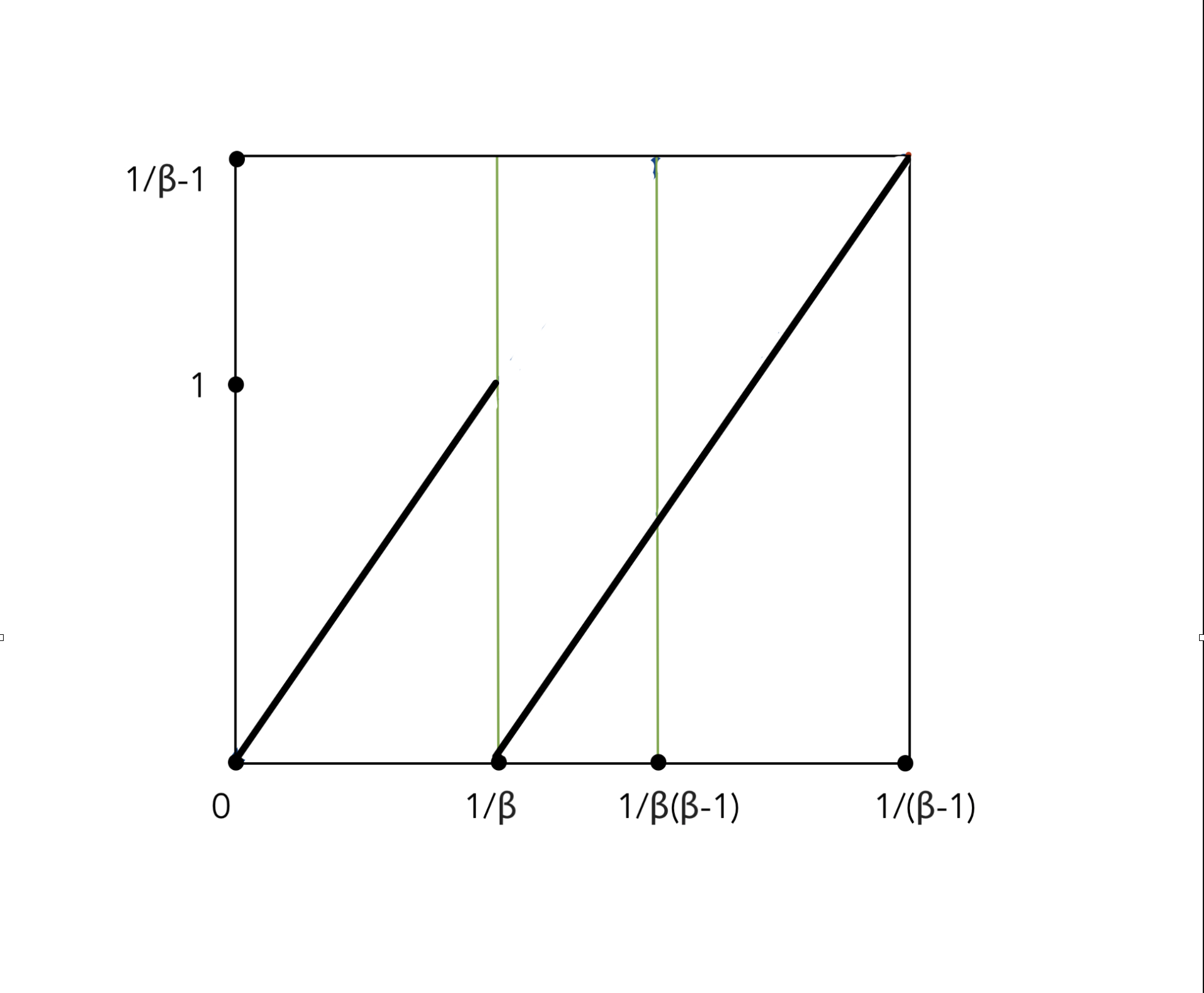}
\caption{Diagram of {$T(x)$}}
\label{fig:lion}
\end{figure}

As one can observe from the graph, and the formula for $a_i$
\[
\left\{
\begin{array}{ll}
     T^{(i-1)}(x) \in [0,\frac{1}{\beta}) & \Rightarrow a_i = 0\\
     T^{(i-1)}(x) \in [\frac{1}{\beta},\frac{1}{\beta-1}]   &\Rightarrow a_i=1\\
\end{array} 
\right.
\] 
We can also see that the largest value $T(x)$ can take is $\frac{1}{\beta-1}$. We will use this fact in the proof of the following lemma to show that \textit{Algorithm 2} generates expansion for $x$.

\begin{lemma}
Let $(a_i)_{i=1}^{\infty}$ be a coefficient expansion of $x \in [0,\frac{1}{\beta-1}]$ generated by \textit{Algorithm 2}. Then $x = \sum_{i=1}^{\infty}\frac{a_i}{\beta^i}$.
\end{lemma}
\begin{proof}
We will use the first $n$ elements of our expansion to show that, for all $n \in \mathbb{N}$, the sum of fractions is less than or equal to $x$. That is,
\[
\frac{a_1}{\beta^1}+\frac{a_2}{\beta^2}+\dots + \frac{a_n}{\beta^n} \leq x
\]
Then we will show that, as $n \to \infty$, the sum converges to $x$.
Firstly,
\begin{equation*} 
\begin{split}
x &\geq \frac{a_1}{\beta^1}+\frac{a_2}{\beta^2}+\dots + \frac{a_n}{\beta^n}\\
&= \frac{\beta x-T(x)}{\beta^1}+\frac{\beta x-T(x)}{\beta^1}+\frac{\beta T^{(2)}(x)-T^{(3)(x)}}{\beta^n}+\dots + \frac{\beta T^{(n-1)}(x)-T^{(n)}(x)}{\beta^n}\\
&=x - \frac{T(x)}{\beta}+\frac{T(x)}{\beta}-\frac{T^{(2)}(x)}{\beta^2}+\frac{T^{(2)}(x)}{\beta^2}-\frac{T^{(3)}(x)}{\beta^3}+\dots+\frac{T^{(n-1)}(x)}{\beta^{n-1}}-\frac{T^{(n)}(x)}{\beta^n}\\
&=x - \frac{T^{(n)}(x)}{\beta^n}
\end{split}
\end{equation*}
Since $\frac{T^{(n)}(x)}{\beta^n} \geq 0$, we can conclude that the sum of first $n$ elements is always less than or equal to $x$.
\newline
Now, we need to show that as we add all elements, the sum converges to $x$. In fact, that is equivalent to saying  $\frac{a_1}{\beta^1}+\frac{a_2}{\beta^2}+\dots + \frac{a_n}{\beta^n} \to x$ as $n \to \infty$. Therefore, if we can show that
\[
lim_{n \to \infty} x - \frac{T^{(n)}(x)}{\beta^n} = x
\]
we are done. In fact, as we discussed above, the maximum value $T(x)$ can take is $\frac{1}{\beta-1}$.
\[
0 \leq T(x) \leq \frac{1}{\beta-1} \textit{ for all } x \in [0,\frac{1}{\beta-1}]
\]
Using this fact above as a base case, by induction we can also see that, for all $n \in \mathbb{N}$,  $0 \leq T^{(n)}(x) \leq \frac{1}{\beta-1}$ since $x, T(x) \in [0,\frac{1}{\beta-1}]$. Hence,
\[
0 \leq T^{(n)}(x) \leq \frac{1}{\beta -1}
\]
Then. by dividing each side by $\beta^n$
\[
\frac{0}{\beta^n} = 0 \leq \frac{T^{(n)}(x)}{\beta^n} \leq \frac{1}{\beta^n(\beta-1)}
\]
Using the Sandwich Theorem for limits
\[
lim_{n \to \infty} 0 \leq lim_{n \to \infty}\frac{T^{(n)}(x)}{\beta^n} \leq lim_{n \to \infty}\frac{1}{\beta^n(\beta-1)}
\]
\[
0 \leq lim_{n \to \infty}\frac{T^{(n)}(x)}{\beta^n} \leq 0
\]
Hence, $lim_{n \to \infty}\frac{T^{(n)}(x)}{\beta^n} = 0$. So $lim_{n \to \infty} x - \frac{T^{(n)}(x)}{\beta^n} = x$. Hence, we conclude that the sum converges to $x$ as $n\to \infty$.Hence, $x = \sum_{i=1}^{\infty}a_i\beta^{-i}$
\end{proof}
So using \textit{Algorithm 2} we get the coefficient sequence with only 1's and 0's, given as $$\{a_i\}_{i=1}^{\infty} \quad \text{for each}\quad  x\in [0,\frac{1}{\beta-1}].$$
\newline

\section{Different versions of Algorithm 2}
In this section, we will modify Algorithm 2 by replacing $T(x)$ with slightly different functions, and see what kind of sequences we could get. As we mentioned earlier, almost aThis suggest that an expansion of a some $x \in [0, \frac{1}{\beta-1}]$ might not be unique in the base-$\beta$. As an example of different functions to replace $T(x)$, we could use : 
\[
R^{(i)}(x)=   \left\{
\begin{array}{ll}
      \beta R^{(i-1)}(x) &  R^{(i-1)}(x)\leq \frac{1}{\beta(\beta-1)} \\
      \beta R^{(i-1)}(x)-1 &  R^{(i-1)}(x)> \frac{1}{\beta(\beta-1)} \\
\end{array} 
\right. 
\]for $i\geq 1$ and where $R^{(i)}(x) = \underbrace{(R \circ R \circ \dots \circ R)}_{\text{i times}}(x)$ with $R^{(0)}(x)=x$.
\newline

Here we denote the elements of coefficient expansions generated using $R(x)$ by  $b_i$'s.To find a particular $b_i$ we need to calculate 
\[
b_i=\beta R^{(i-1)}(x)-R^{(i)}(x).
\]
Following is the diagram of $R(x)$

\begin{figure}[H]
\centering
\includegraphics[width=10cm]{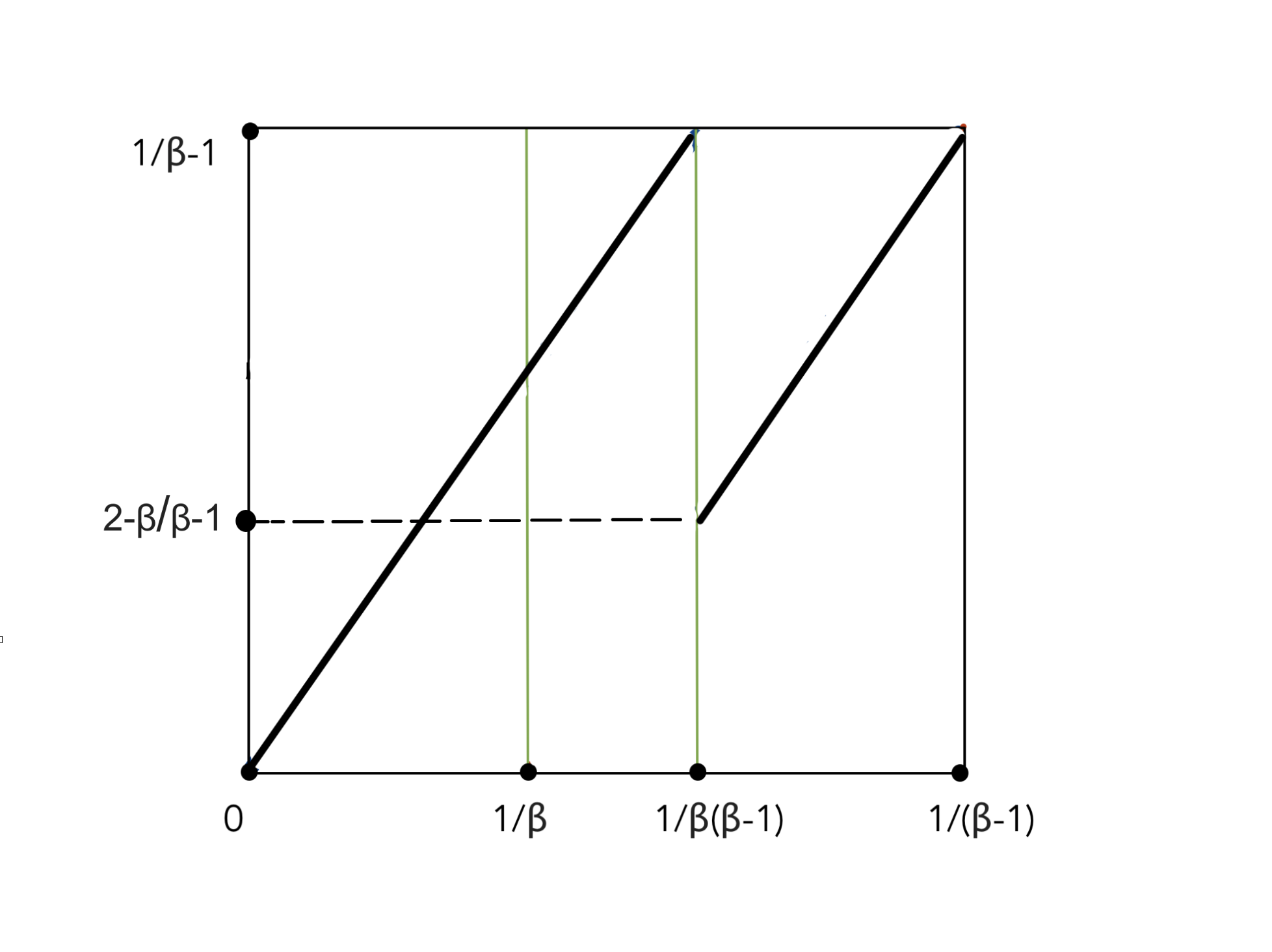}
\caption{Diagram of {$R(x)$}}
\label{fig:lion}
\end{figure}
As one can observe from the graph, and formula for $b_i$
\[
\left\{
\begin{array}{ll}
     R^{(i-1)}(x) \in [0,\frac{1}{\beta(\beta-1)}] & \Rightarrow b_i = 0\\
     R^{(i-1)}(x) \in (\frac{1}{\beta(\beta-1)},\frac{1}{\beta-1}]   &\Rightarrow b_i=1\\
\end{array} 
\right.
\] 
for all $i \in \mathbb{N}$.
\newline
We can use the similar arguments that we used in the proof of \textit{Algorithm 2} (\textit{Lemma 1.3}) to show that $x = \sum_{i=1}^{\infty}b_i\beta^{-i}$
\newline
\chapter{Different $\beta$-expansions for $x$}
As we discussed earlier, we might have different $\beta$-expansions for a particular $x$. In this chapter, we will look into generating them using different methods in different sections. In the first section we will generate combine $R(x)$ and $T(x)$ in a diagram, and show that on sometimes we are free to choose $0$ or $1$ for the some $a_i$ in the coefficient sequence of $x$. Hence, we can have different expansions of it by making those choices different each time. In the second section, we will look for words of length $n$ on the coefficient sequence , and see when we can replace them with the different words of same length. For example, if $\beta = \frac{1+\sqrt{5}}{2}$, then we can replace $100$ with $011$ in the coefficient sequence. In fact, using this method we will show that almost all $x \in [0,\frac{1}{\beta-1}]$ have uncountably infinite $beta$-expansion.
\section{Generating $\beta$-expansions using both $T(x)$ and $R(x)$}
We already know that we generate the $\beta$-expansion of $x$ using \textit{Algorithm 2} with $T(x)$ and $R(x)$ seperately. In this section we will try to use them both at the same time to generate coefficient sequence. As they overlap on some regions of $[0,\frac{1}{\beta-1}]$, we are justified to do this. We will see that for non-overlapping regions, we are free to use $R(x)$ or $T(x)$. That means we are free to choose $a_i$ to be $0$ or $1$, if we lay on the non-overlapping region while expanding $x$. Following is the diagram of $T(x)$ and $R(x)$.
\begin{figure}[htp]
\centering
\includegraphics[width=8cm]{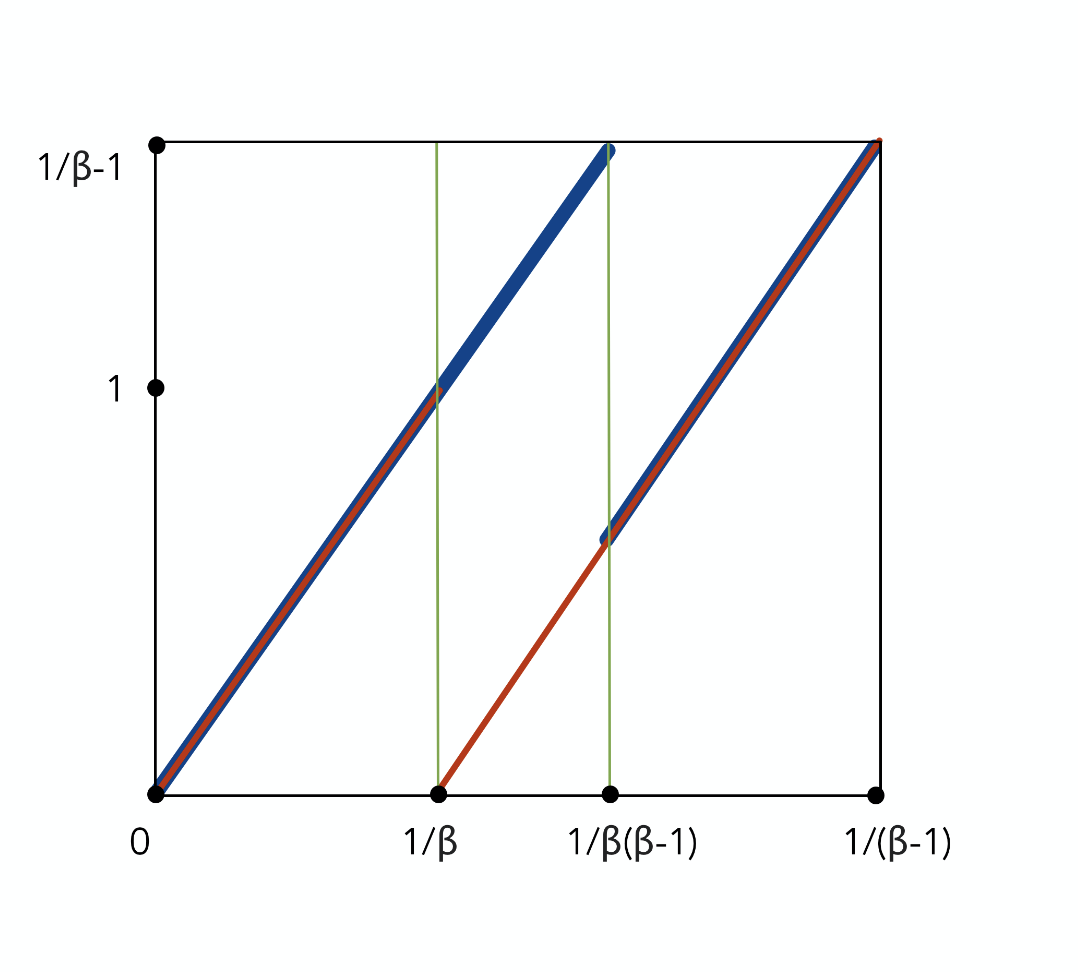}
\caption{Diagram of \textcolor{red}{$T(x)$}, and \textcolor{blue}{$R(x)$}}
\label{fig:lion}
\end{figure}
\newline
We divide the region $[0,\frac{1}{\beta-1}]$ into 3 non-overlapping regions, as follows
\[
[0,\frac{1}{\beta-1}] = [0,\frac{1}{\beta}) \cup [\frac{1}{\beta},\frac{1}{\beta(\beta-1)}] \cup (\frac{1}{\beta(\beta-1)},\frac{1}{\beta-1}]
\]
\newline
On the first region $[0,\frac{1}{\beta})$, $T^{(i)}(x)$ and $R^{(i)}(x)$ overlap also recall from the Section 1.2 and 1.3
\[
T^{(i)}(x) = \beta T^{(i-1)}(x)-1 \textit{ for all } i \in \mathbb{N}
\]
\[
R^{(i)}(x) = \beta R^{(i-1)}(x) \textit{ for all } i \in \mathbb{N}
\]
Hence, $a_i = b_i = 0$ if we lay on this region at some step i.
\newline
On the second region $[\frac{1}{\beta},\frac{1}{\beta(\beta-1)}]$, $T^{(i)}(x)$ and $R^{(i)}(x)$  do not overlap , also recall from the Section 1.2 and 1.3
\[
T^{(i)}(x) = \beta T^{(i-1)}(x) \textit{ for all } i \in \mathbb{N}
\]
\[
R^{(i)}(x) = \beta R^{(i-1)}(x) \textit{ for all } i \in \mathbb{N}
\]
Hence, $a_i = 1$ and $b_i = 0$ if we lay on this region at some step i. Therefore for the $i$'th element of our expansion we are free to choose $0$ or $1$.
\newline
On the third region $(\frac{1}{\beta(\beta-1)},\frac{1}{\beta-1}]$, $T^{(i)}(x)$ and $R^{(i)}(x)$ overlap also recall from the Section 1.2 and 1.3
\[
T^{(i)}(x) = \beta T^{(i-1)}(x)-1 \textit{ for all } i \in \mathbb{N}
\]
\[
R^{(i)}(x) = \beta R^{(i-1)}(x)-1 \textit{ for all } i \in \mathbb{N}
\]
Hence, $a_i = b_i = 1$ if we lay on this region at some step i.
\newline
Note that, if we choose $0$ when we are in the middle region, on the next step we will be pushed to right on the diagram.(Not necessarily to the right region though, depending on $\beta$ we might still be on the middle region, but definitely on the right of the current point) That is because we need to compensate the $0$'s in the sequence by putting $1$'s after that. However, if we choose $1$, on the next step we will be pulled left because we need to add $0$'s afterwards to balance out the $1$'s. (Again not necessarily to the left region as before).
\newline
Hence we conclude that any coefficient sequence $(c_i)_{i=1}^{\infty}$ for $x$ that we obtain using this method will satisty $x = \sum_{i=1}^{\infty}c_i\beta^{-i}$.
\section{Replacing words of length $n$}
\ \ As we mentioned in the beginning of the chapter, for some $\beta$ we have option to replace a specific words of length $n$, with different words of the same length. In this section we will start talking about swapping words, when $\beta$ is golden ratio, that is $\beta=\frac{1+\sqrt{5}}{2}$. We will prove that we can swap $100$ with $011$ or other way around, when $\beta$ is golden ratio. Then we will look into other words of length $3,4$ and try to find out for which $\beta \in (1,2)$ we have an option to swap them. 
\newline
\ \ Note that, the words that can be replaced with each other should be different from each other in each element. For example $011$ and $100$ is differs from each other in each element, while $101$ and $001$ have the same element in the second place.
\newline
\ \ Let's choose $\beta = \frac{1+\sqrt(5}{2}$ \textit{(Golden Ratio)}, and let  $\{a_i\}_{i=1}^{\infty}$ be generated by a fair coin toss, where $x=\sum^{\infty}_{i=1}a_i \beta^{-i}$. By the following result, we will prove that we can replce $100$ with $011$ when $\beta = \frac{1+\sqrt{5}}{2}$.
\begin{lemma}
Let $\beta = \frac{1+\sqrt{5}}{2}$, and $x$ is given as above. Then if we swap $100$ with $011$ in $\{a_i\}_{i=1}^{\infty}$, we still get on coefficient sequence of $x$.
\end{lemma}
\begin{proof}
We write the coefficient sequence of $x$, as infinite series.
Then any $100$ on our expansion, have the following form in our series:
$$\frac{1}{\beta^{k}}+\frac{0}{\beta^{k+1}}+\frac{0}{\beta^{k+2}}$$ for some $k\in \mathbb{N}$. 

Then our claim becomes
$$\frac{1}{\beta^{k}}+\frac{0}{\beta^{k+1}}+\frac{0}{\beta^{k+2}} = \frac{0}{\beta^{k}}+\frac{1}{\beta^{k+1}}+\frac{1}{\beta^{k+2}}$$
$$\Leftrightarrow \frac{1}{\beta}=\frac{1}{\beta^2}+\frac{1}{\beta^3}$$
$$\Leftrightarrow \beta^2-\beta-1=0$$
$$\Leftrightarrow \beta = \pm \frac{1+\sqrt{5}}{2} \quad (Golden \, Ratio)$$

Hence, if $\beta$ is Golden ratio, we can swap $100$ with $011$ on our expansion sequence.
\end{proof}
Later, we will use this lemma and prove that such $x$, whose coefficient sequence is generated by a fair coin toss, has uncountably infinite different $\beta$-expansions when $\beta$ is \textit{Golden Ratio}.
\newline
Now, we will look into more general cases. Let $\{\alpha_i\}_{i=1}^{\infty}$ be an coefficient sequence for given $x$ and $\{\gamma_i\}_{i=1}^{\infty}$ be different sequence obtained by swapping words ${a_1}{a_2}\dots{a_n}$ in $\{\alpha_i\}_{i=1}^{\infty}$ with ${b_1}{b_2}\dots{b_n}$, where $a_i \not= b_i$ for all $i$.
\newline
\newline
If $n=3,4$, we have the following options of swapping words, which might for particular $\beta \in (1,2)$:
$$n=3 \Rightarrow  \left\{
\begin{array}{ll}
      100_{\beta}=011_{\beta} \\
      010_{\beta}=101_{\beta} \\
      001_{\beta}=110_{\beta} \\
\end{array} 
\right. $$

$$n=4 \Rightarrow  \left\{
\begin{array}{ll}
      1000_{\beta}=0111_{\beta} \\
      0100_{\beta}=1011_{\beta} \\
      0010_{\beta}=1101_{\beta} \\
      0001_{\beta}=1110_{\beta} \\
      1100_{\beta}=0011_{\beta} \\
      1010_{\beta}=0101_{\beta} \\
      1001_{\beta}=0110_{\beta} \\
\end{array} 
\right. $$
\newline
\newline
To find for which $\beta$, we can use these swapping words, we associate each of them with polynomial, and solve for $\beta$. If the associated polynomial doesn't have root in $(1,2)$, we say such $\beta$ doesn't exist. We will solve associated polynomial for some of the swapping words above. We have used \cite{klijn_2012} online calculator to find the roots of the following polynomials.
\begin{itemize}

\item \begin{align*}
100_{\beta}=011_{\beta}\quad & {\Rightarrow} \quad \frac{1}{\beta} = \frac{1}{\beta^2}+\frac{1}{\beta^3} \\
                             & {\Rightarrow}\quad \beta^2-\beta-1 = 0\\
                             & {\Rightarrow}\quad \beta = \frac{1+\sqrt{5}}{2} \quad (Golden \, Ratio)\\
\end{align*}
\item \begin{align*}
010_{\beta}=101_{\beta}\quad & {\Rightarrow} \quad \frac{1}{\beta^2} = \frac{1}{\beta}+\frac{1}{\beta^3} \\
                             & {\Rightarrow}\quad \beta^2-\beta+1 = 0\\
                             & {\Rightarrow}\quad \textit{No real $\beta$ exist satisfying this equation.}\\
\end{align*} 
\item \begin{align*}
001_{\beta}=110_{\beta}\quad & {\Rightarrow} \quad \frac{1}{\beta^3} = \frac{1}{\beta^1}+\frac{1}{\beta^2} \\
                             & {\Rightarrow}\quad \beta^2+\beta-1 = 0\\
                             & {\Rightarrow}\quad \textit{No $\beta$ exist on (1,2) satisfying this equation.}\\
\end{align*}
\item \begin{align*}
1000_{\beta}=0111_{\beta}\quad & {\Rightarrow} \quad \frac{1}{\beta} = \frac{1}{\beta^2}+\frac{1}{\beta^3}+\frac{1}{\beta^4} \\
                             & {\Rightarrow}\quad \beta^3-\beta^2-\beta-1 = 0\\
                             & {\Rightarrow}\quad \beta \approx 1.83929\\
                             \end{align*}
\item \begin{align*}
0100_{\beta}=1011_{\beta}\quad & {\Rightarrow} \quad \frac{1}{\beta^2} = \frac{1}{\beta}+\frac{1}{\beta^3}+\frac{1}{\beta^4} \\
                             & {\Rightarrow}\quad \beta^3-\beta^2+\beta+1 = 0\\
                             & {\Rightarrow}\quad \textit{No $\beta$ exist on (1,2) satisfying this equation.}\\
\end{align*}
\item \begin{align*}
0010_{\beta}=1101{\beta}\quad & {\Rightarrow} \quad \frac{1}{\beta^3} = \frac{1}{\beta}+\frac{1}{\beta^2}+\frac{1}{\beta^4} \\
                             & {\Rightarrow}\quad \beta^3+\beta^2-\beta+1 = 0\\
                             & {\Rightarrow}\quad \textit{No $\beta$ exist on (1,2) satisfying this equation.}\\
\end{align*}
\item \begin{align*}
0001{\beta}=1110{\beta}\quad & {\Rightarrow} \quad \frac{1}{\beta^4} = \frac{1}{\beta^1}+\frac{1}{\beta^2}+\frac{1}{\beta^3} \\
                             & {\Rightarrow}\quad \beta^3-\beta^2-\beta-1 = 0\\
                             & {\Rightarrow}\quad \textit{No $\beta$ exist on (1,2) satisfying this equation.}\\
\end{align*}
\item \begin{align*}
1100{\beta}=0011{\beta}\quad & {\Rightarrow} \quad \frac{1}{\beta} + \frac{1}{\beta^2}=\frac{1}{\beta^3}+\frac{1}{\beta^4} \\
                             & {\Rightarrow}\quad \beta^3+\beta^2-\beta-1 = 0\\
                            & {\Rightarrow}\quad \textit{No $\beta$ exist on (1,2) satisfying this equation.}\\
\end{align*}
\item \begin{align*}
1010{\beta}=0101{\beta}\quad & {\Rightarrow} \quad \frac{1}{\beta} + \frac{1}{\beta^3}=\frac{1}{\beta^2}+\frac{1}{\beta^4} \\
                             & {\Rightarrow}\quad \beta^3-\beta^2+\beta-1 = 0\\
                             & {\Rightarrow}\quad \textit{No $\beta$ exist on (1,2) satisfying this equation.}\\
\end{align*}
\item \begin{align*}
1001{\beta}=0110{\beta}\quad & {\Rightarrow} \quad \frac{1}{\beta} + \frac{1}{\beta^4}=\frac{1}{\beta^2}+\frac{1}{\beta^3} \\
                             & {\Rightarrow}\quad \beta^3-\beta^2-\beta+1 = 0\\
                             & {\Rightarrow}\quad \textit{No $\beta$ exist on (1,2) satisfying this equation.}\\
\end{align*}
\end{itemize}
As we see from the calculations above, for the words of length $3,4$, we can only replace $100$ with $011$ when $\beta=\frac{1+\sqrt{5}}{2}$ and $1000$ with $0111$ when $\beta \approx 1.83929$. 
\chapter{Set of expansion of a given $x$ with the base of golden ratio.}
\ \ In this chapter we will talk about expansions in the base \textit{Golden Ratio}, i.e $\beta=\frac{1+\sqrt{5}}{2}$. First we will show that the probability of having infinitely many $100$'s and $011$'s in an expansion generated by infinite coin toss ($1$ for Head, $0$ for Tails) is $1$. Then we will use this fact to prove the main result of this chapter, which is to show that $x$, whose coefficient sequence is generated by a fair coin toss, has uncountably infinite different $\beta$-expansions. We will first start with introducing some notations that we will use along the chapter.
\newline
\ \ Let $\mathbb{S}_{\beta}(x)$ be the set of expansions of $x \in [0,\frac{1}{\beta-1}]$ with the base of $\beta \in (1,2)$. 
$$\mathbb{S}_{\beta}(x) = \{\{e_i\}_{i=1}^{\infty}: e_i \in \{0,1\} \quad and \quad \sum_{i=1}^{\infty}\frac{e_i}{\beta^i}=x\}.$$
\ \ Again,let's choose $\beta = \frac{1+\sqrt(5}{2}$ \textit{(Golden Ratio)}, and let  $\{a_i\}_{i=1}^{\infty}$ be generated by a fair coin toss, where $x=\sum^{\infty}_{i=1}a_i \beta^{-i}$. By the following result, we will prove that we can replce $100$ with $011$ when $\beta = \frac{1+\sqrt{5}}{2}$.
\begin{lemma}
Let $\{a_i\}_{i=1}^{\infty}$ be coefficient sequence that is generated by a fair coin toss. Then the probability of having infinitely many $100s$ and $011s$ on our sequence is $1$.
\end{lemma}
\begin{proof}
If we partition our given sequence by words of length 3, we get the following
\[ 
x = \underbrace{|a_1 a_2 a_3 |}_{\text{1}}\underbrace{|a_4 a_5 a_6 |}_{\text{2}}\underbrace{|a_7 a_8 a_9 |}_{\text{3}}\underbrace{|a_{10} a_{11} a_{12} |}_{\text{4}}\underbrace{\dots}_{\text{$\rightarrow \infty$}}
\]
If we denote the event of getting $100_{\beta}$ or $011_{\beta}$ on the $i^{th}$ word by $E_i$, then we have $P(E_i) = \frac{1}{4}$.
\newline
\newline
Since we have infinite number of $E_i$'s and each of them is independent with the sum of probabilities $\sum^{\infty}_{i=1}P(E_i)=\sum^{\infty}_{i=1}\frac{1}{4}=\infty$, converse $Borel-Cantelli$ lemma says that probability of infinitely many of $E_i$'s occuring is $1$.
\end{proof}
In fact our previous result takes into account $100_{\beta}$'s and $001_{\beta}$'s that may occur on a certain positions (starting from position $3k+1; k \in \mathbb{N}$) of our expansion sequence. Applying similar argument we can fill the gap and show that $100_{\beta}$ or $011_{\beta}$'s appearing at any position of our sequence are infinitely many.
\newline
\begin{lemma}
Let $x = \sum_{i=1}^{\infty} a_i \beta^{-i}$, where $\{a_i\}_{i=1}^{\infty}$ be generated by fair coin toss, and $\beta$ is Golden Ratio. Then $\mathbb{S}_{\beta}(x)$ is uncountably infinite set.
\end{lemma}
\begin{proof}

To show this set is uncountable, we are getting back to our previous setting, and chop a given sequence $x=(a_1a_2a_3a_4a_5a_6a_7a_8a_9a_{10}\dots)$ into words of length 3. Then, without loss of generality, we blank out all capsules with $100_{\beta}$'s or $011_{\beta}$'s as follows:
$$ x =\dots |\hrectangle|a_{k+1}a_{k+2}a_{k+3}|a_{k+4}a_{k+5}a_{k+6}|\dots \dots|\hrectangle|\dots$$
\newline
\newline
Define an isomorphism $g:\{0,1\}\rightarrow\{100_{\beta},011_{\beta}\}$ such that $g(0)=100_{\beta}$ and $g(1) = 011_{\beta}$.
\newline
\newline
Define map $f:\{0,1\}^{\mathbb{N}}\rightarrow\mathbb{S}_{\beta}(x)$ given as 
\[
f(\{a_i\}_{i=1}^{\infty}) =\dots g(a_1)a_{k+1}a_{k+2}a_{k+3}a_{k+4}a_{k+5}a_{k+6} \dots g(a_2). 
\]
That is, since each $a_i$ could only take $0$ or $1$, we can put the value of $g(a_i)$ in the $i^{th}$ word.
\newline
\newline
To prove that $\mathbb{S}_{\beta}(x)$ is uncountable, we just need to show $f$ is injective. Then, since $\{0,1\}^{\mathbb{N}}$ is uncountable, we can conclude that $\mathbb{S}_{\beta}(x)$ is uncountable too.
\newline
To show that $f$ is injective assume $\{p_i\}_{i=1}^{\infty}, \{q_i\}_{i=1}^{\infty} \in \{0,1\}^{\mathbb{N}}$ such that $\{p_i\}_{i=1}^{\infty}$ and  $\{q_i\}_{i=1}^{\infty}$ differs in the $k^{th}$ place. That is $\{p_i\}_{i=1}^{\infty}\neq \{q_i\}_{i=1}^{\infty}$ Then clearly, $g(p_k)\neq g(q_k)$ and hence, $f(\{p_i\}_{i=1}^{\infty})\neq f(\{q_i\}_{i=1}^{\infty}).$ So we conclude that $f$ is injective. Since $\{0,1\}^{\mathbb{N}}$ is uncountable, we say $\mathbb{S}_{\beta}(x)$ is uncountable as well.
\end{proof}
\chapter{Greedy and Lazy Expansions} 
In this and the next chapters, our main focus will be the Pal Erdos's paper and mainly talk about \textit{Theorem 1}. This paper characterises the unique expansions of $1 \in \mathbb{Z} $ But first we need to familiarize ourselves with the definitions and results that are building blocks of the \textit{Theorem 1} that we will prove in the first section of this chapter. Throughout the paper, the letter $q$ is used to denote the number we refer as $\beta$, but we will stick to our notation throughout this paper. First, we need to familirise ourselves with the working terminology.
\begin{definition}[Lexicographic order]
Let $(a_i)$, $(b_i)$ be the sequences of real numbers. We say $(a_i)$ is lexicographically less than $(b_i)$, denoting $(a_i) \stackrel{L}{<}(b_i)$, if there exist positive integer $n \in \mathbb{N}$ such that 
\[\left\{
\begin{array}{ll}
        i<n \Rightarrow  a_i = b_i  \\
        i = n \Rightarrow  a_i < b_i \\
\end{array} 
\right. 
\]
\end{definition}
 
\begin{definition}[Lazy Expansion]
Let $(a_i) \in \{0,1\}^{\mathbb{N}}$ be a beta expansion of given $x\in [0,\frac{1}{\beta-1}]$. We say $(a_i)$ is lazy expansion of $x$, if it's lexicographically smallest expansion of $x$.
\end{definition}
Note that, by definition, in the \textit{lazy expansion} we always try to put $0$ in the $i$'th position, and compensate it by putting $1$'s on the all positions after that. Hence, we can formulate its $i$'th element in the following way:
\[a_i = 
\left\{
\begin{array}{ll}
        0 \qquad \text{    if  } \ \ \sum_{j=1}^{i-1}a_j\beta^{-j}+  \sum_{j=i+1}^{\infty}\beta^{-j} \geq x\\
        1 \qquad  \text{    if  }\ \  \sum_{j=1}^{i-1}a_j\beta^{-j}+  \sum_{j=i+1}^{\infty}\beta^{-j} < x \\
\end{array} 
\right. 
\]

\begin{definition}[Greedy expansion]
Let $(a_i) \in \{0,1\}^{\mathbb{N}}$ be a beta expansion of given $x\in [0,\frac{1}{\beta-1}]$. We say $(a_i)$ is greedy expansion of $x$, if it's lexicographically largest expansion of $x$.
\end{definition}
Again, note that, by definition, in the \textit{greedy expansion} we always try to put $1$'s in the $i$'th position. If the sum of gets larger than $x$ when we put $1$, we put $0$ instead. Hence, we can formulate $i$'th element in the \textit{lazy expansion} as follows:
m'th element, for all $m\geq 1$, is generated as follows.
\[a_i = 
\left\{
\begin{array}{ll}
        0 \qquad \text{    if  } \ \ \sum_{j=1}^{i-1}a_i\beta^{-j}+  \beta^{-i} > x\\
        1 \qquad  \text{    if  }\ \  \sum_{j=1}^{i-1}a_i\beta^{-j}+ \beta^{-i} \leq  x \\
\end{array} 
\right. 
\]

\begin{definition} [Unique expansion]
We say $(a_i) \in \{0,1\}^{\mathbb{N}}$ is a unique expansion of given $x\in [0,\frac{1}{\beta-1}]$, if $x$ has only 1 $\beta$ expansion
\end{definition}
Recall that, if $\beta$ and $x$ are given, we denote the set of $\beta$-expansions of $x$ by $\mathbb{S}_{\beta}(x)$. Then, 
\[
max(\mathbb{S}_{\beta}(x)) = \textit{ greedy expansion of }x
\]
\[
min(\mathbb{S}_{\beta}(x)) = \textit{ lazy expansion of }x
\]
Here $max(\mathbb{S}_{\beta}(x))$ and $min(\mathbb{S}_{\beta}(x))$ denote lexicographically largest and smallest element in $\mathbb{S}_{\beta}(x)$ correspondingly. Hence, if $\mathbb{S}_{\beta}(x)$ has only one element, then $max(\mathbb{S}_{\beta}(x))=min(\mathbb{S}_{\beta}(x))$. That is unique expansion of $x$ is both lazy and greedy.
\newline
If $(a_i)$ is unique expansion of $x$, then it's both greedy and lazy expansion of it. That is because
\chapter{Characterisation of greedy expansion of 1}
In this chapter, we are going to characterise the greedy expansion of $1$. We'll formulate the characterisation statement, at the end of this chapter as a Theorem 1. Before that, we will state 4 different results from Erdos' paper to use later in the proof of Theorem 1.
\begin{lemma}[Lemma 1]Let $x = \sum_{i=1}^{\infty}a_{i}\beta^{-i}$, where $a_i \in \{0,1\}$
\newline
\newline
a) We say that this expansion is greedy if and only if

$$\sum_{i=1}^{\infty}a_{k+i}\beta^{-i} < 1 \ \ \ \forall k \in \mathbb{N} \ such \  that  \ a_k = 0$$
\newline
b)We say that this expansion is lazy if and only if $$\sum_{i=1}^{\infty}(1-a_{k+i})\beta^{-i} < 1 \ \ \ \forall k \in \mathbb{N} \ such \ that\ a_k = 1$$
\end{lemma}
\begin{proof}
a) $ (\Rightarrow) $: Assume $\sum_{i=1}^{\infty}a_{i}\beta^{-i}$ is greedy and let $k \in \mathbb{N}$ such that $a_k = 0$.We will use the definition of greedy algorithm to show the inequality.
\begin{equation*} 
\begin{split}
a_k = 0 & \Rightarrow \sum_{i<k}a_i \beta^{-i} + \beta^{-k} > x \\
& \Rightarrow \sum_{i<k}a_i \beta^{-i} + \beta^{-k} > \sum_{i=1}^{\infty}a_{i}\beta^{-i} \\
& \Rightarrow \sum_{i<k}a_i \beta^{-i} + \beta^{-k} > \sum_{i<k}a_{i}\beta^{-i}+ a_k\beta^{-k}+\sum_{i>k}a_i\beta^{-i} \\
& \Rightarrow  \beta^{-k} > \sum_{i=k+1}^{\infty}a_{i}\beta^{-i} \\
\textit{(mutiplying both sides by } \beta^k) &\Rightarrow 1 > \sum_{i=k+1}^{\infty}a_{i}\beta^{-i+k}\\
\textit{(changing indices)} &\Rightarrow  \sum_{i=1}^{\infty}a_{k+i}\beta^{-i} < 1
\end{split}
\end{equation*}
\newline
\newline
$ (\Leftarrow) $: Assume $\sum_{i=1}^{\infty}(a_{k+i})\beta^{-i} < 1 \ \ \ \forall k \in \mathbb{N} \ such \ that\ a_k = 0$. To get the contradiction, suppose $\sum_{i=1}^{\infty}a_{i}\beta^{-i}$ is not greedy. Then there exist another expansion of $x$ such that $x = \sum_{i=1}^{\infty}c_i\beta^{-i}$, where $c_i \in \{0,1\}$ and $(a_i)\stackrel{L}{<}(c_i)$. By the definition of the lexicographic order, there exist $m \in \mathbb{N}$ such that
\begin{equation*} 
\begin{split}
i<m & : a_i = c_i \\
i = m & : a_i<c_i \\
\end{split}
\end{equation*}
This implies by the choice of $a_i$ and $c_i$ that $a_m = 0$ and $c_m = 1$. 
\newline

Now we will show that, the existence of the expansion $(c_i)$ contradicts our initial assumption that $\sum_{i=1}^{\infty}(a_{k+i})\beta^{-i} < 1 \ \ \ \forall k \in \mathbb{N} \ such \ that\ a_k = 0$.
\begin{equation*} 
\begin{split}
x &= \sum_{i=1}^{\infty}a_{i}\beta^{-i} = \sum_{i=1}^{\infty}c_{i}\beta^{-i}\\
& \Rightarrow \sum_{i=m+1}^{\infty}a_{i}\beta^{-i} = \beta^{-m} + \sum_{i=m+1}^{\infty}c_{i}\beta^{-i} \geq \beta^{-m}\\
& \Rightarrow \sum_{i=m+1}^{\infty}a_{i}\beta^{-i+m} \geq 1\\
& \Rightarrow \sum_{i=1}^{\infty}a_{m+i}\beta^{-i} \geq 1 \ \textit{for some } a_k = 0\\
\end{split}
\end{equation*}
This gives us a contradiction. Hence, there doesn't exist another expansion that is lexicographically greater than $(a_i)$.
\newline
\newline
b) $(\Rightarrow):$Assume $x=\sum_{i=1}^{\infty}a_i\beta^{-i}$ is lazy, and let $k \in \mathbb{N}$ such that $a_k=1$. Again, we will use the definition of the lazy expansion to show the inequality.
\begin{equation*} 
\begin{split}
a_k=1 & \Rightarrow \sum_{i<k}a_{i}\beta^{-i} + \sum_{i>k}\beta^{-i} <x\\ 
& \Rightarrow \sum_{i>k}\beta^{-i} < \beta^{-k}+\sum_{i>k}a_i\beta^{-i}\\
& \Rightarrow \sum_{i>k}(1-a_i)\beta^{-i} < \beta^{-k}\\
& \Rightarrow \sum_{i>k}(1-a_i)\beta^{-i+k} < 1\\
& \Rightarrow \sum_{i=k+1}^{\infty}(1-a_i)\beta^{-i+k}<1\\
& \Rightarrow \sum_{i=1}^{\infty}(1-a_{k+i})\beta^{-i}<1\\
\end{split}
\end{equation*}

$ (\Leftarrow) $: Assume $\sum_{i=1}^{\infty}(1-a_i)\beta^{-i}<1 \ \ \forall k \in \mathbb{N} \textit{ such that } a_k = 1$. To get a contradiction, suppose $x=\sum_{i=1}^{\infty}a_i\beta^{-i}$ is not lazy. Then there exist another expansion of $x$ such that $x = \sum_{i=1}^{\infty}w_i\beta^{-i}$,where $w_i \in \{0,1\}$ which is lexicographically less than $(a_i)$: $(a_i) \stackrel{L}{>}(w_i)$. By the definition of the lexicographic order, there exist $m \in \mathbb{N}$ such that 
\begin{equation*} 
\begin{split}
i<m & : a_i = w_i \\
i = m & : a_i>w_i \\
\end{split}
\end{equation*}
This implies by the choice of $a_i$ and $w_i$ that $a_m = 1$ and $w_m = 0$. 
\newline

Now, since $a_m=1$, we can use our assumption. 
\begin{equation*} 
\begin{split}
x &= \sum_{i=1}^{\infty}a_{i}\beta^{-i} = \sum_{i=1}^{\infty}w_{i}\beta^{-i}\\
& \Rightarrow \sum_{i=m+1}^{\infty}a_{i}\beta^{-i} + \beta^{-m} = \sum_{i=m+1}^{\infty}w_{i}\beta^{-i} \leq \sum_{i>m}\beta^{-i}\\
& \Rightarrow \sum_{i>m}(1-a_i)\beta^{-i} \geq \beta^{-m}\\
& \Rightarrow \sum_{i=m+1}^{\infty}(1-a_{i})\beta^{-i+m} \geq 1\\
& \Rightarrow \sum_{i=1}^{\infty}(1-a_{m+i})\beta^{-i} \geq 1 \ \textit{for some } a_m = 1\\
\end{split}
\end{equation*}
This contradicts our initial assumption, showing there is some $a_k=1$, for which the inequality is not satisfied. Hence, we conclode that there does not exist an expansion, which is lexicographically less than $(a_i)$ So $x=\sum_{i=1}^{\infty}a_i\beta^{-i}$ is lazy.
\newline
\newline
This completes the proof of the lemma.
\end{proof}

\begin{lemma}[Lemma 2]
Let $x \geq 1$:
\newline
a)If $x=\sum_{i=1}^{\infty}a_i\beta^{-i}$ is greedy, then $(a_{k+i})\stackrel{L}{<}(a_i)$ for all $k \in \mathbb{N}$ such that $a_k = 0$.
\newline
\newline
b)If $x=\sum_{i=1}^{\infty}a_i\beta^{-i}$ is unique, then 
\[
(1-a_{k+i})\stackrel{L}{<}(a_i) \text{ for all } k \in \mathbb{N} \text{ such that }a_k = 1
\]
and 
\[
(a_{m+i})\stackrel{L}{<}(a_i) \text{ for all } m \in \mathbb{N} \text{ such that }a_m = 0
\]
\end{lemma}

\begin{proof}
a) Assume $x=\sum_{i=1}^{\infty}a_i\beta^{-i}$ is greedy and let $k\in \mathbb{N}$ such that $a_k = 0$. Then by Lemma 1, we have
\begin{equation*}
\begin{split}
& \sum_{i=1}^{\infty}a_{k+i} \beta^{-i}  < 1 \leq x=\sum_{i=1}^{\infty}a_i\beta^{-i} \\
& \Rightarrow \sum_{i=1}^{\infty}(a_i-a_{k+1})\beta^{-i} > 0\ \ \ \ \ (1)\\
\end{split}
\end{equation*}
This immidiately implies that $(a_i)\stackrel{L}{\neq}(a_{k+i})$, as the left hand side is strictly greater than $0$. Now, to show $(a_{k+i})\stackrel{L}{<}(a_i)$, we only need to show that $(a_{k+i})\stackrel{L}{>}(a_i)$ gives contradiction. 
\newline
Assume $(a_{k+i})\stackrel{L}{>}(a_i)$. Then there exist $m \in \mathbb{N}$ such that 
\begin{equation*} 
\begin{split}
i<m & : a_i = a_{k+i} \\
i = m & : a_i<a_{k+i} \\
\end{split}
\end{equation*}
Hence we can write $(1)$ as $\sum_{i=m}^{\infty}(a_i-a_{k+i})\beta^{-i}$, since first $m-1$ terms cancel each other out.
Also, by the choice of $a_i$, we conclude that $a_m = 0$ and $a_{k+m} = 1$. Since $x=\sum_{i=1}^{\infty}a_i\beta^{-i}$ is greedy, we now show that
\begin{equation*}
\begin{split}
a_m = 0 &\Rightarrow \sum_{i=1}^{m-1}a_i\beta^{-i} + \beta^{-m} > \sum_{i=1}^{\infty}a_i\beta^{-i}\\
&\Rightarrow \beta^{-m} > \sum_{i=m+1}^{\infty} a_i \beta^{-i} \ \ \ \ \ \ \ \ (2)\\
\end{split}
\end{equation*}
We use this result along with $a_m = 0$ and $a_{k+m} = 1$in $(1)$, 

$$a_m\beta^{-m} + \sum_{i=m+1}^{\infty}a_i\beta^{-i} - (a_{k+m}\beta^{-m} + \sum_{i=m+1}^{\infty}a_{k+i}\beta^{-i}) > 0$$
\newline
$$ \text{(Using }a_m = 0 \text{ and } \text{(2))} \Rightarrow \beta^{-m} - \sum_{i=m+1}^{\infty}a_{k+i}\beta^{-i}-\beta^{-m} > 0$$
$$ \Rightarrow \sum_{i=m+1}^{\infty}a_{k+i}\beta^{-i} < 0$$
This is contradiction, hence such $m \in \mathbb{N}$ cannot exist and 
$(a_{k+i})\stackrel{L}{<}(a_i)$.
\newline
\newline
b) By the definition, if $(a_i)$ is unique, then it's both lazy and greedy. Since it's greedy,$(a_{m+i})\stackrel{L}{<}(a_i)$ is already satisfied whenever $a_m = 0$. We only need to check if $(1-a_{k+i})\stackrel{L}{<}(a_i)$ holds, whenever $a_k=1$. We'll use argument by contradiction to show this.
\newline
\newline
Assume there exist $k \in \mathbb{N}$ such that $a_k=1$ but $ (1-a_{k+i})\stackrel{L}{<}(a_i)$ is not satisfied. Then either $ (1-a_{k+i})\stackrel{L}{=}(a_i)$ or $ (1-a_{k+i})\stackrel{L}{>}(a_i)$.
\newline
\newline
Assume $ (1-a_{k+i})\stackrel{L}{=}(a_i)$. In this case we have
\[
\sum_{i=1}^{\infty}(1-a_{k+i})\beta^{-i} = \sum_{i=1}^{\infty}a_{i}\beta^{-i} = x \geq 1
\]
This contradicts the assumption that $(a_i)$ is lazy, as by Lemma 10.1(b) $\sum_{i=1}^{\infty}(1-a_{k+i})\beta^{-i}<1$ whenever $a_k = 1$. Hence, $ (1-a_{k+i})\stackrel{L}{\neq}(a_i)$.
\newline
\newline
Now, assume $ (1-a_{k+i})\stackrel{L}{>}(a_i)$. Then by the definition of lexicographic order, there exist $n \in \mathbb{N}$ such that 
\[\left\{
\begin{array}{ll}
      For \ i<n &  1-a_{k+i} = a_i \qquad \qquad \qquad \qquad \qquad \qquad \qquad   (1)\\
      For \  i = n &  1-a_{k+n}>a_{i} \Leftrightarrow a_{n} = 0 \ \ and \ \ 1-a_{k+n} = 1 \ \ \ \ \ \  (2)\ \  \\
\end{array} 
\right. 
\]
Since, the expansion is unique, we have 
\[
\sum_{i=1}^{\infty}(1-a_{k+i})\beta^{-i} < 1 \leq x
\]
This implies that there exist an expansion $(1-a_{k+i})$ whose $n'th$ term equals to $1$ $(1-a_{k+n} = 1)$, but still this expansion is less than $x$, while for $(a_i)$ $a_n = 0$. Hence we can find another expansion of $x$, say $(w_i)$,
\[
x = \sum_{i = 1}^{\infty}a_i\beta^{-i} = \sum_{i = 1}^{\infty}w_i\beta^{-i}
\]
such that $w_i = a_i$ for $i<n$, but for $i=n$ $w_n = 1$ and $a_n = 0$. Then $(a_{i})\stackrel{L}{\neq}(w_i)$. To show the existence, of $(w_i)$ we just need to show that $\frac{1}{\beta-1} \geq x - \sum_{i=1}^{n}w_i\beta^{-i} \geq 0$. Then we can generate the rest of the elements of $(w_i)$ using algorithms we have already constructed.
\newline
\newline
Since $ x \leq \frac{1}{\beta-1} $ by assumption, it is guaranteed that 
\[
x - \sum_{i=1}^{n}w_i\beta^{-i} \leq \frac{1}{\beta-1} - \sum_{i=1}^{n}w_i\beta^{-i} < \frac{1}{\beta-1}
\]
To show $x - \sum_{i=1}^{n}w_i\beta^{-i} \geq 0$, we again use 
\[
\sum_{i=1}^{\infty}(1-a_{k+i})\beta^{-i} < 1 \leq x
\]
We can write it as
\[
 0 \leq x - \sum_{i=1}^{\infty}(1-a_{k+i})\beta^{-i} \leq x - \sum_{i=1}^{n}(1-a_{k+i})\beta^{-i}
\]
By construction,
\[\left\{
\begin{array}{ll}
      For \ i<n &  1-a_{k+i} = a_i = w_i \qquad \qquad \qquad \qquad \qquad \qquad \qquad   (1)\\
      For \  i = n &  w_i = 1-a_{k+n}>a_{i} \Leftrightarrow a_{n} = 0 \ \ and \ \ w_n = 1-a_{k+n} = 1 \ \ \ \ \ \  (2)\ \  \\
\end{array} 
\right. 
\]
Hence, 
\[
 \sum_{i=1}^{n}(1-a_{k+i})\beta^{-i} = \sum_{i=1}^{n}(w_i)\beta^{-i}
\]
So we can use it in above inequality to show that 
\[
0 \leq x - \sum_{i=1}^{\infty}(1-a_{k+i})\beta^{-i} \leq x - \sum_{i=1}^{n}(1-a_{k+i})\beta^{-i}
 = x - \sum_{i=1}^{n}(w_i)\beta^{-i}
\]
\[
0 \leq x - \sum_{i=1}^{n}(w_i)\beta^{-i}
\]
Hence, there exist an expansion $(w_i)$ of $x$ such that $(a_{i})\stackrel{L}{\neq}(w_i)$. This contradicts the assumption that $(a_i)$ is unique. So $ (1-a_{k+i})\stackrel{L}{\not >}(a_i)$. Hence $ (1-a_{k+i})\stackrel{L}{<}(a_i)$
\end{proof}
\begin{lemma}
Let $x = \sum_{i=1}^{\infty}a_i\beta^{-i}\leq 1$ and $y = \sum_{i=1}^{\infty}d_i\beta^{-i}$ for $y \in [0,\frac{1}{\beta-1}]$ such that $(d_{k+i})\stackrel{L}{<}(a_i)$ whenever $d_k = 0$. We say that if $\sum_{i=1}^{\infty}a_i\beta^{-i}$ is infinite or $\sum_{i=1}^{\infty}d_i\beta^{-i}$ is finite then $\sum_{i=1}^{\infty}d_i\beta^{-i}$ is greedy expansion of $y$.

\end{lemma}
\begin{proof}
We are going to use $(d_{k+i})\stackrel{L}{<}(a_i)$ whenever $d_k = 0$ to show that $\sum_{i=k+1}^{\infty}d_i\beta^{-i} < \beta^{-k}$ which is equivalent to the condition in Lemma 10.1(a) for an expansion to be greedy.
\newline
\newline
We first need to show the existence of a sequence of integers $k_1<k_2<k_3<\dots$ such that 
\[
\sum_{i=k_j+1}^{k_{j+1}}d_i\beta^{-i} \leq \beta^{-k_j}-\beta^{-k_{j+1}} \ \ \  \forall j \in \mathbb{N}
\]
Set $k_1$ such that $d_{k_1} = 0$. Then $(d_{k_1+i})\stackrel{L}{<}(a_i)$ implies that there exist $m_1 \in \mathbb{N}$ such that 
\[\left\{
\begin{array}{ll}
      For \ i<m_1 &  d_{k_1+i} = a_i \qquad \qquad \qquad \qquad \qquad \qquad \qquad  \ (1)\\
      For \  i = m_1 &  d_{k_1+m_1}<a_{m_1} \Leftrightarrow d_{k_1+m_1} = 0 \ \ and \ \ a_{m_1} = 1 \ \ \ \ \  (2)\ \  \\
\end{array} 
\right. 
\]
Using (1) above, we can say that 
\[
\sum_{i=k_1+1}^{k_1+m_1}d_i\beta^{-i} = \sum_{i=1}^{m_1-1}a_i\beta^{-i-k_1} = \beta^{-k_1}\bigg (\sum_{i=1}^{m-1}a_i\beta^{-i}\bigg )
\]
We can write $\sum_{i=1}^{m_1-1}a_i\beta^{-i} \leq 1- \sum_{i=m_1}^{\infty}a_i\beta^{-i}$, as $x = \sum_{i=1}^{\infty}a_i\beta^{-i} \leq 1$
\[
\beta^{-k_1}\bigg (\sum_{i=1}^{m_1-1}a_i\beta^{-i}\bigg ) \leq \beta^{-k_1}\bigg (1-\sum_{i=m_1}^{\infty}a_i\beta^{-i}\bigg )
\]
Since, we already know that $a_{m_1} = 1$, we can write 

\[
\beta^{-k_1}\bigg(1-\sum_{i=m_1}^{\infty}a_i\beta^{-i}\bigg )\leq \beta^{-k_1}\big (1-a_{m_1}\beta^{-m_1}\big ) = \beta^{-k}\big (1-\beta^{-m_1}\big ) = \beta^{-k} - \beta^{-(k+m_1)}
\]
So we can combine these all.
\[
\sum_{i=k_1+1}^{k_1+m_1}d_i\beta^{-i} \leq \beta^{-k_1} - \beta^{-(k_1+m_1)}  
\]
Now, set $k_1+m_1 = k_2$, hence
\[
\sum_{i=k_1+1}^{k_2}d_i\beta^{-i} \leq \beta^{-k_1} - \beta^{-k_2} 
\]
Since $d_{k_1+m_1} = d_{k_2} = 0$, we can repeat the same process using $(d_{k_2+i})\stackrel{L}{<}(a_i)$ for $d_{k_2} = 0$, and find $m_2 \in \mathbb{N}$ such that
\[\left\{
\begin{array}{ll}
      For \ i<m_2 &  d_{k_2+i} = a_i \qquad \qquad \qquad \qquad \qquad \qquad \qquad  \ (1)\\
      For \  i = m_2 &  d_{k_2+m_2}<a_{m_2} \Leftrightarrow d_{k_2+m_2} = 0 \ \ and \ \ a_{m_2} = 1 \ \ \ \ \  (2)\ \  \\
\end{array} 
\right. 
\]
Again, setting $d_3 = d_2+m_2$, we can conclude that
\[
\sum_{i=k_2+1}^{k_3}d_i\beta^{-i} \leq \beta^{-k_2} - \beta^{-k_3} 
\]
Since, we already know that for any $j \in \mathbb{N}$ such that $d_{k_j} = 0$, the $(d_{k_j+i})\stackrel{L}{<}(a_i)$ is satisfied, we can always find $m_j\in \mathbb{N}$, such that $d_{k_j+m_j} = 0$, as
\[\left\{
\begin{array}{ll}
      For \ i<m_j &  d_{k_j+i} = a_i \qquad \qquad \qquad \qquad \qquad \qquad \qquad  \ (1)\\
      For \  i = m_j &  d_{k_j+m_j}<a_{m_j} \Leftrightarrow d_{k_j+m_j} = 0 \ \ and \ \ a_{m_j} = 1 \ \ \ \ \  (2)\ \  \\
\end{array} 
\right. 
\]
Hence, denoting $d_{j+1} = d_j+m_j$ we can generate all the elements of our sequence $k_1<k_2<k_3< \dots <d_j< \dots$
\newline
\newline
Denoting $k=k_1$ we can conclude that, 
\[
\sum_{i>k}d_i\beta^{-i} = \sum_{j=1}^{\infty}\sum_{i = k_j+1}^{k_{j+1}}d_i\beta^{-i} \leq \sum_{j=1}^{\infty}(\beta^{-k_j}-\beta^{-k_{j+1}}) = \beta^{-k}
\]
Hence,
\[
\sum_{i>k}d_i\beta^{-i}  \leq \beta^{-k} 
\]
\[
\Leftrightarrow \sum_{i>k}d_i\beta^{-i+k}  \leq 1
\]
\[
\Leftrightarrow \sum_{i=1}^{\infty}d_{k+i}\beta^{-i}  \leq 1
\]
Now if we can show that, the inequality above becomes strict when either $\sum_{i=1}^{\infty}a_i\beta^{-i}$ is infinite or $\sum_{i=1}^{\infty}d_i\beta^{-i}$ is finite, then our proof is complete.
\newline
\newline

Assume $\sum_{i=1}^{\infty}a_i\beta^{-i}$ is infinite, then the inequality becomes strict, since 
\[
\bigg(1-\sum_{i=m_j}^{\infty}a_i\beta^{-i}\bigg )<\big (1-a_{m_j}\beta^{-m_j}\big )=\big (1-\beta^{-m_j}\big )
\]
for all $j\in \mathbb{N}$ and hence,
\begin{equation*}
\begin{split}
\sum_{i=k_j+1}^{k_{j+1}}d_i\beta^{-i} &= \beta^{-k_j}\bigg ( \sum_{i=1}^{m_j-1}a_i\beta^{-i}\bigg)\\
&\leq \beta^{-k_j}\bigg(1-\sum_{i=m_j}^{\infty}a_i\beta^{-i}\bigg )\\
&<\beta^{-k_j}(1-\beta^{-m_j})\\
&=\beta^{-k_j}-\beta^{-(k_j+m_j)}\\
&=\beta^{-k_j}-\beta^{-k_{j+1}}
\end{split}
\end{equation*}
Hence, 
\[
\sum_{i=1}^{\infty}d_{k+i}\beta^{-i}  < 1
\]
Now assume $\sum_{i=1}^{\infty}d_i\beta^{-i}$ is finite. Then there exist $k_p$ in our sequence such that $d_i = 0$ for $i>k_p$. Then
\begin{equation*}
\begin{split}
\sum_{i=k_1+1}^{k_p}d_i\beta^{-i} &= \sum_{j=1}^{p}\sum_{i = k_j+1}^{k_{j+1}}d_i\beta^{-i} \\
&\leq \sum_{j=1}^{p}(\beta^{-k_j}-\beta^{-k_{j+1}}) \\
&=\beta^{-k_1} - \beta^{-k_{p+1}}\\
&<\beta^{-k_1}=\beta^{-k}
\end{split}
\end{equation*}
Hence, again,
\[
\sum_{i=k+1}^{\infty}d_i\beta^{-i}  < \beta^{-k} 
\]
\[
\Leftrightarrow \sum_{i=k+1}^{\infty}d_i\beta^{-i+k}  < 1
\]
\[
\Leftrightarrow \sum_{i=1}^{\infty}d_{k+i}\beta^{-i}  < 1
\]
This completes the proof of the lemma.
\end{proof}
\begin{lemma}[Lemma 4]
Let $x = \sum_{i=1}^{\infty} a_i\beta^{-i} \leq 1$.
\newline
a) If $(a_{k+i})\stackrel{L}{<}(a_i)$ whenever $a_k = 0$, then $(a_i)$ is greedy expansion.
\newline
b) If $(1-a_{k+i})\stackrel{L}{<}(a_i)$ whenever $a_k = 1$, then $(a_i)$ is lazy expansion.
\end{lemma}
\begin{proof}
a) This directly comes from Lemma 10.3. Set $y = x$ and $(d_i) = (a_i)$ in the lemma 3. Then $y$ has expansion $(d_i) = (a_i)$ and assume 
\[
(d_{k+i}) \stackrel{L}{=}(a_{k+i})\stackrel{L}{<}(a_i).
\]
That is equivalent to our assumtion in this lemma such that $(a_{k+i})\stackrel{L}{<}(a_i)$. Now to apply Lemma 10.3, we need extra condition that either $(d_i)$ is finite, or $(a_i)$ is infinite. Since, $(d_i)=(a_i)$ this condition is always satisfied, as $(a_i)$ is either finite or infinite. Hence, we conclude that $(d_i) = (a_i)$ is greedy expansion.
\newline
\newline
b) \ \ First we'll show that if $0<x<1$, then $x$ has infinite expansion. 
\newline
\ \ If $x=0$, then expansion is unique. $a_i = 0$ for all $i \in \mathbb{N}$. Let $0<x<1$, and assume $(a_i)$ is finite. Then there exist $k \in \mathbb{N}$, such that $a_i = 0$ for all $i>k$ and $a_k = 1$. Then $(1-a_{k+i}) = 1$ for all $i \in \mathbb{N}$. Then $(1-a_{k+i})\stackrel{L}{<}(a_i)$ is not satisfied for such $k$, since $(a_i) = 0$ for some $i$, but $(1-a_{k+i})=1$ for all $i$. This contradicts our assumption that $(1-a_{k+i})\stackrel{L}{<}(a_i)$ whenever $a_k = 1$. So expansion of $0<x<1$ is always infinite under this assumtion.
\newline
\ \ Now we can apply Lemma 10.3, similarly to part (a), setting  $d_i = (1-a_i)$. Then whenever $a_k=1$, we have $1-a_k = d_k = 0$. This allows as to apply Lemma 10.3, as
\[
(d_{k+i})\stackrel{L}{<}(a_i) \text{ whenever } d_k=0
\]
is equivalent to
\[
(1-a_{k+i})\stackrel{L}{<}(a_i) \text{ whenever } a_k=1
\]
Hence, we can conclude that $(d_i)$ is greedy expansion. And since $(d_i) = (1-a_i)$, this implies that $(a_i)$ is lazy expansion.
\end{proof}
\begin{theorem}[Theorem 1]Let $1 = \sum_{i=1}^{\infty}a_i\beta^{-i}$.
\newline
a) $(a_i)$ is greedy expansion of $1$ if and only if
\[
(a_{k+i})\stackrel{L}{<}(a_i) \text{ whenever } a_k=0
\]
b) $(a_i)$ is unique expansion of $1$ if and only if
\[
(1-a_{k+i})\stackrel{L}{<}(a_i) \text{ whenever } a_k=1
\]
\end{theorem}
\begin{proof}
a) $(\Rightarrow)$ By Lemma 10.2(a), since $ 1 = x = \sum_{i=1}^{\infty}a_i\beta^{-i}\leq 1$, and $(a_i)$ is greedy, we can conclude that $(a_{k+i})\stackrel{L}{<}(a_i) \text{ whenever } a_k=0$.
\newline
\newline
$(\Leftarrow)$ By Lemma 10.4(a), since $ 1 = x = \sum_{i=1}^{\infty}a_i\beta^{-i}\leq 1$ and $(a_{k+i})\stackrel{L}{<}(a_i) \text{ whenever } a_k=0$, we can conclude that $(a_i)$ is greedy.
\newline
\newline
b) $(\Rightarrow)$By Lemma 10.2(b), since $ 1 = x = \sum_{i=1}^{\infty}a_i\beta^{-i}\leq 1$, and $(a_i)$ is unique, we can conclude that $(1-a_{k+i})\stackrel{L}{<}(a_i) \text{ whenever } a_k=1$.
\newline
\newline
$(\Leftarrow)$ By Lemma 10.4(b), since $ 1 = x = \sum_{i=1}^{\infty}a_i\beta^{-i}\leq 1$ and $(1-a_{k+i})\stackrel{L}{<}(a_i) \text{ whenever } a_k=1$, we can conclude that $(a_i)$ is unique.
\end{proof}

\chapter{Characterisation of lazy expansion of 1}
In this chapter, we present our solution for the first open problem at the end of Erdos paper. 
\begin{lemma}
Let $x = \sum_{i \geq 1} a_i \beta^{-i} \geq 1$ and $\beta \in (\frac{1+\sqrt5}{2},2)$. $x = \sum_{i \geq 1} a_i \beta^{-i}$ is lazy if and only if $(1-a_{k+1})<(a_i)$ whenever $a_k = 1$.
\end{lemma}
\begin{proof}
$(\Rightarrow)$
If $\beta \in (\frac{1+\sqrt5}{2},2)$, then $a_1 = 1$ as $$\sum_{i>1}\beta^{-i} \leq 1 \leq x$$.
\newline
\newline
We claim that if $x = \sum_{i \geq 1} a_i \beta^{-i}$ is lazy expansion, then $(1-a_{k+i})<(a_i)$. We will use arguments for contradiction.
\newline
\newline
First assume $(1-a_{k+i})=(a_i)$. By lemma 1, we know that if $x = \sum_{i \geq 1} a_i \beta^{-i}$ is lazy, then $\sum_{i \geq 1} (1-a_{k+i}) \beta^{-i}<1 \leq x =  \sum_{i \geq 1} a_i \beta^{-i} $ whenever $a_k = 1$. However, $\sum_{i \geq 1}(a_i)-\sum_{i \geq 1}(1-a_{k+i})\beta^{-i} > 0$. The strict inequality shows that, these two expansions cannot be equal. Hence, this case is not possible.
\newline
\newline
Now assume $(1-a_{k+i})>(a_i)$, whenever $a_k=1$. Since we know for sure that $a_1 = 1$, let $k=1$.Then $$  \left\{
\begin{array}{ll}
      a_1 = 1 &  (1-a_{1+1}) = 1 \\
      a_2 = 0 &  (1-a_{1+2}) = 0 \text{ or } 1 \\
\end{array} 
\right. $$
If $1-a_{1+2} = 1$, then $a_3 = 0$. However, 
\begin{equation*}
\begin{split}
x = \sum_{i\geq 1}a_i\beta^{-i} &\leq \frac{1}{\beta}+\frac{1}{\beta^4}+\frac{1}{\beta^5}+\frac{1}{\beta^6}+\dots\\
&= \frac{1}{\beta}+ \sum_{i>3}\frac{1}{\beta^i}\\
&= \frac{1}{\beta}+ \frac{1}{\beta^3(\beta-1)} (1)\\
\end{split}
\end{equation*}
This sum (1) is strictly less that $1$ for all $\beta \in (\frac{1+\sqrt{5}}{2},2)$.
This contradicts our initial assumption that $x \geq 1$. Hence $a_3 = 1$ and $(1-a_{1+2}) = 0$. Adding these to our expansion list: 
$$  \left\{
\begin{array}{ll}
      a_1 = 1 &  (1-a_{1+1}) = 1 \\
      a_2 = 0 &  (1-a_{1+2}) = 0 \\
      a_3 = 1 &  (1-a_{1+3}) = 1 \\
      a_4 = 0 &  (1-a_{1+4}) = 0 \text{ or } 1
\end{array} 
\right. $$
Again, if $1-a_{1+4} = 1$, then $a_5 = 0$. However, 
\begin{equation*}
\begin{split}
x = \sum_{i\geq 1}a_i\beta^{-i} &\leq \frac{1}{\beta}+\frac{1}{\beta^3}+\frac{1}{\beta^6}+\frac{1}{\beta^7}+\frac{1}{\beta^8}+\dots\\
&= \frac{1}{\beta}+ \frac{1}{\beta^3}+\sum_{i>5}\frac{1}{\beta^i}\\
&= \frac{1}{\beta}+ \frac{1}{\beta^3}+\frac{1}{\beta^5(\beta-1)} (2)\\
\end{split}
\end{equation*}
This sum (2) is strictly less that $1$ for all $\beta \in (\frac{1+\sqrt{5}}{2},2)$.
This contradicts our initial assumption that $x \geq 1$. Hence $a_5 = 1$ and $(1-a_{1+4}) = 0$. Adding these to our expansion list: 
$$  \left\{
\begin{array}{ll}
      a_1 = 1 &  (1-a_{1+1}) = 1 \\
      a_2 = 0 &  (1-a_{1+2}) = 0 \\
      a_3 = 1 &  (1-a_{1+3}) = 1 \\
      a_4 = 0 &  (1-a_{1+4}) = 0 \\
      a_5 = 1 &  (1-a_{1+5}) = 1 \\
      a_6 = 0 & \dots
\end{array} 
\right. $$
Continuing this process, we already see that there is a pattern emerging while generating the elements of both expansions assuming $(1-a_{k+i})>(a_i)$. That is $a_m = 0$ for all even $m \in \mathbb{N}$, which in turn gives us a choice for choosing $(1-a_{1+m})$ to be $0$ or $1$. If we can show that, we can never have $1-a_{1+m} = 1$ for all even $m \in \mathbb{N}$, then we can conclude that for $k=1$, we have $a_k = 1$, but $(1-a_{k+i})>(a_i)$ not satisfied by the definition of lexicographic order.
\newline
We will show that for $1-a_{1+m} =0$ for all even $m \in \mathbb{N}$ by induction.
We have already shown that this is true for $m = 2$ and $m=4$ as a base case. 
\newline
Now assume the statement is true for the first $n$  natural numbers, where $n$ is even. Then  
$$ (*) \left\{
\begin{array}{ll}
      \dots \\
      a_n = 0 &  (1-a_{1+n}) = 0 \\
      a_{n+1} = 1 &  (1-a_{1+(n+1)}) = 1 \\
      a_{n+2} = 0 &  (1-a_{1+(n+2)}) = 0 \text{ or } 1\\
\end{array} 
\right. $$

Using the above, we will show that $(1-a_{1+(n+2)}) = 0$. 
\newline 
Assume for contradiction $(1-a_{1+(n+2)}) = 1$, then $a_{n+3} = 0$. Hence we can use this for finding an upper bound for $x$.

\begin{equation*}
\begin{split}
x = \sum_{i\leq 1}\frac{a_i}{\beta^i} 
&\leq \frac{1}{\beta}+\frac{1}{\beta^3}+\dots +\frac{1}{\beta^{n+1}} + \frac{1}{\beta^{n+4}}+\frac{1}{\beta^{n+5}}+\dots\\
&= (\frac{1}{\beta}+\frac{1}{\beta^3}+\dots +\frac{1}{\beta^{n+1}}) + \sum_{i>n+3}\frac{1}{\beta^i}\\
& = (B) = \frac{1}{\beta-1} - (\frac{1}{\beta^2}+\frac{1}{\beta^4}+\dots+\frac{1}{\beta^{n+2}})-\frac{1}{\beta^{n+3}} \\
\end{split}
\end{equation*}
If we can show that $1 > $ (B) for $\beta = \frac{1+\sqrt{5}}{2}$, then we are done, as this would imply that $x<1$.
Observe that 
\[
1-(B) = 1-\frac{1}{\beta-1} + (\frac{1}{\beta^2}+\frac{1}{\beta^4}+\dots+\frac{1}{\beta^{n+2}})+\frac{1}{\beta^{n+3}} > 0
\]
This is true, because for $\beta \in (\frac{1+\sqrt{5}}{2}, 2) $ $\frac{1+\sqrt{5}}{2} < 1$, and the rest of the terms are positive. Hence, we conclude $1>$ (B) which is the contradiction we are looking for. Thus, $(1-a_{1+(n+2)}) = 0$ and $(a_{(n+2)}) = 1$. This completes inductive step.
\newline
Hence $(1-a_{k+i})>(a_i)$ is not satisfied for $a_1 = 1$. Hence this case is also not true.
\newline
Then we can conclude that $(1-a_{k+i})<(a_i)$, hence compelete the proof of $(\Rightarrow)$.

$(\Leftarrow)$: Lemma 4 (b) directly applies here, as $x=1 \leq 1$. Hence completes the proof.

\end{proof}
\begin{lemma}
Let $1 = \sum_{i \geq 1} a_i \beta^{-i}$ and $\beta \in (1,\frac{1+\sqrt5}{2}]$. If $\sum_{i \geq 1} a_i \beta^{-i}$ is lazy then $(a_{k+1})>(a_i)$ whenever $a_k = 0$.
\end{lemma}
\begin{proof}
$(\Leftarrow)$: Since, $\beta \in (1,\frac{1+\sqrt5}{2}]$, we know that $a_1$ = 0. Then assume for contradiction $(a_{k+i})=(a_i)$ whenever $a_k = 0$. Since, $a_1$ = 0 we can choose $k=1$. Then this gives the following list.
$$  \left\{
\begin{array}{ll}
      a_1 = 0 &  (a_{1+1}) = 0 \\
      a_2 = 0 &  (a_{1+2}) = 0 \\
      a_3 = 0 &  (a_{1+3}) = 0 \\
      a_4 = 0 &  (a_{1+4}) = 0 \\
      a_5 = 0 &  (a_{1+5}) = 0 \\
      a_6 = 0 & \dots
\end{array} 
\right. $$
As we assumed $(a_{1+i})=(a_i)$, this gives us contradiction since $\sum_{i \geq 1} a_i \beta^{-i} = 0 \neq 1$.
Now, assume that $(a_{k+i})<(a_i)$ whenever $a_k=0$. Then again,
$$  \left\{
\begin{array}{ll}
      a_1 = 0 &  (a_{1+1}) = 0 \\
      a_2 = 0 &  (a_{1+2}) = 0 \\
      a_3 = 0 &  (a_{1+3}) = 0 \\
      a_4 = 0 &  (a_{1+4}) = 0 \\
      a_5 = 0 &  (a_{1+5}) = 0 \\
      a_6 = 0 & \dots
\end{array} 
\right. $$
Again this implies that all $a_i$'s are $0$. Hence, this is also contradiction. So, we conclude that $(a_{k+i})>(a_i)$ whenever $a_k=0$.
\end{proof}
Then we can conclude with the following theorem.
\begin{theorem}[Theorem 2]
Let $1=\sum_{i=1}^{\infty}a_i\beta^{-i}$
\newline
a)If $(1-a_{k+i})<(a_i)$ whenever $a_k=1$, then $(a_i)$ is lazy expansion.
\newline
b) If $\beta \in (1,\frac{1+\sqrt5}{2}]$ and $(a_i)$ is lazy, then $(1-a_{k+i})<(a_i)$ whenever $a_k=1$.
\newline
c) If $\beta \in (\frac{1+\sqrt5}{2},2)$ and $(a_i)$ is lazy, then $(a_{k+i})>(a_i)$ whenever $a_k=0$.
\newline
\end{theorem}

\begin{proof}
    Lemmas 6.1 and 6.2, combined with Lemma 4 from the previous chapter gives the result, hence the characterizes lazy expansion of 1. 
\end{proof}

\end{document}